\documentclass{amsproc}
\usepackage{latexsym}
\usepackage{amssymb,amsbsy,amsmath,amsfonts,amssymb,amscd}
\usepackage{epsfig, graphicx,xcolor}

\newtheorem{theorem}{Theorem}[section]
\newtheorem{lemma}[theorem]{Lemma}
\theoremstyle{definition}

\theoremstyle{remark}

\numberwithin{equation}{section}

\begin{document}
\title {Optimal Control of Two-Phase Membrane  Problem  }
\author{Farid Bozorgnia,  Vyacheslav Kungurtsev }
 \address{CAMGSD, Department of Mathematics, Instituto Superior T\'{e}cnico, Lisbon}
\email{farid.bozorgnia@tecnico.ulisboa.pt }
\thanks{F. Bozorgnia was supported by the FCT  fellowship SFRH/BPD/33962/2009 and by Marie Skłodowska-Curie grant agreement No. 777826 (NoMADS)}


\begin{abstract}
We consider an optimal control problem where the state 
  is governed by a  free boundary problem called the two-phase membrane problem and the control appears in the coefficients of the characteristic function of the positivity and negativity parts of the solution.   Our investigation focuses on various properties associated with the control-to-state map.    Due to the non-differentiability of this map, we regularize the state equation.    The existence, uniqueness, and characterization of the optimal pairs are established.
\end{abstract}

\theoremstyle{plain}
\newtheorem{proposition}[theorem]{Proposition}
\newtheorem{corollary}[theorem]{Corollary}
\theoremstyle{definition}
\theoremstyle{definition}
\numberwithin{equation}{section}

\newcommand\R{\mathbf{R}}
\newcommand{\E}{\mathcal{E}}
\newcommand{\C}{\mathcal{C}}
\maketitle

\section{Introduction}

\subsection*{ Notation}
  We will use the following notation throughout the paper.

\begin{itemize}
 \item  $\Omega$ \,   an open, bounded  subset of $\mathbb{R}^{n}$ with smooth boundary.
\item   $\chi_{D}$ \,   the characteristic function of the set $D \subset \mathbb{R}^{n}.$

  \item $\chi^{\varepsilon}_{\mathbb{R}^+}  $ is a non-decreasing, smooth approximation of $\chi_{\mathbb{R}^+} .$
    \item   $f_{\pm} $   positive numbers or positive Lipschitz functions.
       \item  $ u^{\pm} =  \max(\pm u, 0).$
\item ${\{u>0\}}\equiv {\{x\in \Omega: u(x) >0\}}$.
\item $B_{r}(x_{0})$ ball with radius $r$ centered at $x_0$, and  $B_{r}$ centered  at origin. 
 \end{itemize}

\bigskip

Infinite-dimensional optimization problems appear in many practical applications such as minimal surface,
image processing, semiconductor design, and inverse problems. We consider an optimal control problem where the state satisfies a
two-phase membrane problem and the control function appears in the coefficients of the state equation.

To illustrate our motivation, assume we want to adjust the temperature  $u(x)$ in the domain $\Omega$ between two given functions $\theta_{-}, \theta_{+}$  where
 $ \theta_{-}(x) \le  \theta_{+}(x).$  To do this, there are cooling and heating devices that are distributed over domain $\Omega.$
 Due to limitations on the power of devices, the generated heat flux $-\phi \in [-f_{-}, f_{+} ].$ If $u$ is not in the given range then the reformatory heat flux  $-\phi $  will be injected, see the first chapter in \cite{A14} and \cite{A4}   for applications in modeling of population density with high competition.

In recent years there have been intense studies of optimal control with free boundary problems as constraints, see  \cite{A1, A2, B2, C1, A10,   A13}. In the case that the control to state map is the solution of a variational inequality, various authors have investigated this problem \cite{A2, A9, A10, A13}.  The first paper on the optimal obstacle control problem was that
of Adams et al. \cite{A2}. A notable result of that paper is that the optimal obstacle is equal to its corresponding state. The case that  state to control map is  a  semi-linear
elliptic variational inequality and the control function is the obstacle discussed in \cite{B2}.    The authors implement the penalty method to approximate the variational inequality by a sequence of equations, to obtain the optimality system that
characterizes the optimal solution. They obtain an approximate optimality system
and convergence results to pass to the limit in this system.

Also in \cite{A9, A10}  for the optimal control problems governed by elliptic variational inequalities,   optimality conditions are derived, and the existence result for    Lagrange multipliers is proved. Next, a primal-dual active set method is proposed to approximate the optimality systems numerically.

Due to the non-differentiability of the solution operator of the obstacle problem  the
  Bouligand generalized differential operator has been investigated.  In \cite{A15, A16}  the authors characterize and compute a specific
element of the Bouligand generalized differential for a wide class of obstacle problems.

        In \cite{A12}, to approximate the optimal
control problem constrained with elliptic variational inequality, a computational technique based on the pseudo-spectral method is presented. By using the pseudo-spectral method, the infinite-dimensional mathematical
programming with equilibrium constraint, which can be an equivalent form of the considered problem, is converted
to finite-dimensional mathematical programming with complementarity constraint. Later in \cite{A13} authors present a fast computational technique based on the wavelet collocation method for the numerical solution of the mentioned problem.

In this work, we investigate an optimal control problem where the state equation is governed by a free boundary equation called a two-phase membrane.  In  our  control problem, we have the following elements:
\begin{itemize}
\item a) A control  $ \phi  $ which can be chosen from  an admissible set.
\item b) The state of system $u$  to be controlled, which depends on control  $\phi$.
\item c) The state equation that establishes the dependence between the control and the state,
this state equation is a two-phase equation.
\item d) An objective functional to be minimized,    depending on control and state $(\phi, u)$.
\end{itemize}

  Let  $f_{\pm}:\Omega \rightarrow \mathbb{R}$ be non-negative,   bounded,  $g\in W^{1,p}(\Omega)$ for $p>n$  and 
  $g$   changes sign on the boundary.  For    given
$ \phi \in  U_{ad}$  consider  
\begin{equation}\label{eq:state}
\left \{
\begin{array}{ll}
\Delta u= (f_{+}- \phi) \chi_{\{u >0 \}}-(f_{-}+\phi ) \chi_{\{u <0\}} \ & \text{in } \ \Omega,\\
  u=g   &  \text{on }  \  \partial \Omega.
\end{array}
\right.
\end{equation}
 Here   the characteristic  function  of the positive part    is defined as
 \begin{equation*}
 \chi_{\{u >0 \}}=
\left \{
\begin{array}{ll}
1       &    u(x)>0,  \\
  0    &  \text{elsewhere}.
  \end{array}
\right.
\end{equation*}
The admissible set for the control is denoted by $U_{ad},$ defined to be as follows
\[
U_{ad}={\{ \phi \in    L^{2}(\Omega): -f_{-}  \leq \phi  \leq f_{+}}\}.
 \]
 In Subsection \ref{TPM},  we derive   (\ref{eq:state}) and we discuss the existence and uniqueness of the solution of  (\ref{eq:state}). 
 Let $u \equiv T(\phi)$  be the corresponding solution of (\ref{eq:state}) for given $\phi$, i.e., $T$ is the control to state map. Now, we regard  $ \phi \in U_{ad} $ as the control variable and  $ u=T(\phi) $ as the corresponding state variable.
For given $ z \in L^{2}(\Omega),$  we try to find $\phi$ such that the corresponding state $ u$  is close to $z$ in $  L^{2} $. To do this consider the following functional
\begin{equation}\label{Min}
{J(\phi)=\frac{1}{2} \int_{\Omega}  (T(\phi)-z)^{2}\, dx +   \frac{\lambda}{2} \int_{\Omega}  |\phi|^{2}  \, dx}.
\end{equation}
We consider the following minimization
\textbf{Problem (P)}.  Find $ \phi^{*}   \in L^2(\Omega),$  such that
\begin{equation}\label{functional}
J(\phi^{*})=\underset{\phi\in U_{ad}}{\inf\,} J(\phi).
\end{equation}
 Note that  we reduce     the  problem in  $\phi$  by considering  $u= u(\phi) = T(\phi),$
so the problem becomes a nonsmooth and nonconvex problem, which is then hard to tackle.  In particular, the explicit representation of first-order optimality conditions suitable for numerical realization remains an issue.

Although there is extensive literature on optimal control with variational inequality constraints, it's crucial to note that the two-phase membrane problem cannot be reformulated as a variational inequality.
This work aims to study the existence of a minimizer for Problem (P), and first-order optimality condition.  Let us give the outline of the paper. The next section is devoted to derivation of the state equation (\ref{eq:state}). Afterward, we show the existence of a minimizer and prove some properties of the control-to-state map. Section 4 deals with the approximation of the problem by regularizing the state equation and the convergence of the regularized solution.

\section{Background and   Problem Setting}

\subsection{One phase obstacle problem}\label{OPM}

To start, we briefly explain the one-phase obstacle problem and we state a related optimal control problem. Let   $f \in L^{\infty}(\Omega)$ be a non negative function,  $g \in W^{1,2}(\Omega)$ and  $ \varphi \in W^{2, \infty}_{0}(\Omega)$ be given.  Then  minimize the  following  energy functional given by
 \begin{equation}\label{functional1}
\int_{\Omega}\left(\frac{1}{2}|\nabla v|^2+ f \, v  \right)dx,
\end{equation}
over the convex set $K=\{ v \in W^{1,2}(\Omega): v-g \in W^{1,2}_{0}(\Omega): v\ge \varphi \}.$  The minimizer  satisfies the following Complementarity constraints
\begin{equation}\label{one_phase}
\left \{
\begin{array}{ll}
\min(- \Delta v +f , v- \varphi)=0   &  \text{in} \ \Omega, \\
  v=g     & \text{on} \ \partial \Omega.
  \end{array}
\right.
\end{equation}
Let  $v$ be the solution of the  above obstacle problem, then $u=v-\varphi$  is the minimizer of the following functional
 \begin{equation}\label{onephase3}
 \int_{\Omega}\left(\frac{1}{2}|\nabla u|^2+ f_{1} \, u\right)\,dx,
\end{equation}
 over the set
    $K_1=\{ u \in W^{1,2}(\Omega): u-g_{1} \in W^{1,2}_{0}(\Omega): u \ge 0 \},$ with
    \[
    f_1=f-\Delta  \varphi, \quad  g_{1}=g- \varphi.
    \]
The  minimizer of (\ref {onephase3})  satisfies the following one-phase obstacle problem 
 \begin{equation}\label{one_phase2}
\left \{
\begin{array}{ll}
 \Delta u=   (f- \Delta  \varphi)\,      \chi_{\{u >0 \}}  &  \text{in} \ \Omega, \\
  u=g_1     & \text{on} \ \partial \Omega.
  \end{array}
\right.
\end{equation}
We take $\phi= \Delta  \varphi$  as the control and set the map $u=T(\phi).$  
We point out in the literature on optimal control of obstacle problems,  they consider    (\ref{functional1}) or equivalently (\ref{one_phase}) as the sate equation,  which is slightly different from our reformulation given by (\ref{one_phase2}).  To make connection,   $\varphi$ can be recovered  from   $\phi^{*}$  by
\begin{equation}
\left \{
\begin{array}{ll}
  \Delta  \varphi  =    \phi^{*}  &  \text{in} \ \Omega, \\
   \varphi=0     & \text{on} \ \partial \Omega.
  \end{array}
\right.
\end{equation}

\subsection{Two phase  Membrane problem }\label{TPM}

 Assume that the boundary value   $g$  is smooth and changes sign on  $ \partial \Omega$. Let $K=\{ v \in W^{1,2}(\Omega): v=g \, \,  \text{on} \, \, \Omega \}$. Consider the functional
\begin{equation}\label{functional3}
\int_{\Omega}\left(\frac{1}{2}|\nabla v|^2+f_{+} \text {max}(v,0)+ f_{-}\text {max}(-v,0)\right)\, dx,
\end{equation}
 which is convex, weakly lower semi-continuous, and hence attains its infimum at some point  $u \in K$.
The Euler-Lagrange equation corresponding to the minimizer  $u$ is given by   (\ref{twophase1}), which is called \emph{two phase membrane problem}.
\begin{equation}\label{twophase1}
\left \{
\begin{array}{ll}
\Delta u= f_{+} \chi_{\{u >0 \}}-f_{-} \chi_{\{u <0\}}  &  \text{in} \ \Omega, \\
  u=g\     & \text{on} \ \partial \Omega,
  \end{array}
\right.
\end{equation}
 Let $u$ be the solution of  (\ref{twophase1}).  The  \emph{free boundary} is defined  as (see (\cite{A14, A20}))
\[
 \Gamma(u) =\partial \{ x \in \Omega: u(x)>0 \}  \cup \partial\{x\in \Omega: u(x) < 0\} \cap \Omega.
  \]

In Subsection 2.4 we derive a general form of the two-phase membrane problem.  The free boundary  consists of two parts: $$  \Gamma'(u)=\Gamma(u)\cap{\{x \in \Omega :\nabla u(x)= 0 }\}  $$ and  $$ \Gamma''(u)=\Gamma(u)\cap{\{x \in \Omega :  \nabla u(x) \neq 0 }\}. $$

       To see numerical approximations for the two-phase problem we refer to in \cite{AA3, A3, ABP,    A5, A19}. For a substantial reference of free boundary problems of obstacle type, we refer to \cite{A14}.

\subsection{The State    Equation}\label{TOPM}
In this part, we derive a generalized version of (\ref{twophase1})   i.e.,   state equation (\ref{eq:state}).  Consider the functional
\begin{equation}\label{functional2}
I(v)=\int_{\Omega}(\frac{1}{2}|\nabla v|^2+f_{+}(v-\psi)^{+}+ f_{-}(v-\psi)^{-})\, dx,
\end{equation}
with  $v=g$  on   $\partial\Omega$,  $g$ is smooth and  changes sign on the boundary,     and $\psi \in  W^{2,p}$  for all
$p>n$  and  $\psi=0$  on the boundary.  If  $\psi=0$ then functional  (\ref{functional2}) reads  as (\ref{functional3}).  The objective of this section is to show that minimizing   $I(v)$ over $K$  has a solution  $w$  and that  $u=w- \psi$ satisfies equation (\ref{eq:state}).

As  in  Subsection 2.1,   $I$ is coercive,  convex, and weakly lower semi-continuous in $H^{1}(\Omega)$ and so therefore attains its infimum at some point $w.$ The first variation  on the energy $I$ at $ w$,
i.e., considering  $v= w +\varepsilon  \eta $, with $\varepsilon >0$ and
$\eta \in C_0^\infty(\Omega)$ gives
\begin{equation*}
\begin{split}
0 & \leq \frac{I(w +\varepsilon  \eta)- I(w)}{\varepsilon}\\
&=\int_{\Omega} \frac{|\nabla (w +\varepsilon \eta)|^{2} -  |\nabla w |^{2}}{\varepsilon} \,dx +\int_{\Omega}f_{+}\frac{(w+\varepsilon\eta-\psi)^{+}-(w-\psi )^{+}}{\varepsilon} \,dx\\
& +\int_{\Omega}f_{-}\frac{(w+\varepsilon\eta-\psi)^{-}-(w-\psi )^{-}}{\varepsilon} \,dx\\
&:=A+ B+ C.
\end{split}
\end{equation*}
Calculation for $A$ is straightforward. For $B$ reads as,
\begin{equation*}
\begin{split}
B & :=\frac{1} {\varepsilon} \int_{\Omega} f_{+}\,   \left[(w+\varepsilon\eta-\psi)^{+}-(w-\psi )^{+} \right]\,dx\\
& =\frac{1} {\varepsilon} \int_{\{w+\varepsilon\eta > \psi\} \cap \Omega} f_{+}\, (w+\varepsilon\eta-\psi)dx-\frac{1} {\varepsilon} \int_{\{w> \psi\} \cap \Omega} f_{+}\,  (w-\psi)\,dx\\
&= \int_{\{w+\varepsilon\eta > \psi\} \cap \Omega} f_{+}\,  \eta\,dx+\frac{1} {\varepsilon} \int_{\{0\geq w-\psi>-\varepsilon\eta \} \cap \Omega  \cap\{\eta ge 0\}} f_{+}\, (w-\psi)\,dx \\
 & -  \frac{1} {\varepsilon} \int_{\{0\le  w-\psi < -\varepsilon\eta \} \cap \Omega  \cap\{\eta < 0\}} f_{+}\, (w-\psi)\,dx\\\\
&\leq \int_{\{w+\varepsilon\eta > \psi\} \cap \Omega} f_{+} \eta\,dx .
\end{split}
\end{equation*}
The same argument for $C$ as above shows
\begin{equation*}
\int_{\Omega}f_{-}\frac{(w+\varepsilon\eta-\psi)^{-}-(w-\psi )^{-}}{\varepsilon}\, dx\leq - \int_{\{w+\varepsilon\eta < \psi\} \cap \Omega}f_{-}\,  \eta \, dx.
\end{equation*}
Considering the initial inequality we have
\begin{equation*}
\begin{split}
0\leq &\frac{I(w +\varepsilon \eta)- I(w)}{\varepsilon}\\
\leq& \int_{\Omega} \frac{|\nabla (w +\varepsilon \eta)|^{2} -  |\nabla w |^{2}}{2\varepsilon} \, dx+ \int_{\{w+\varepsilon\eta > \psi\} \cap \Omega} f_{+} \, \eta \,dx\\
& -\int_{\{w+\varepsilon\eta < \psi\} \cap \Omega}f_{-}\,  \eta \,dx.
\end{split}
\end{equation*}
Letting $\varepsilon \rightarrow 0 $, and noting,
\begin{equation*}
\int_{\Omega} \frac{|\nabla (w +\varepsilon \eta)|^{2} -  |\nabla w |^{2}}{2\varepsilon} \, dx  =\int_{\Omega}  \nabla w \cdot   \nabla \eta +  \frac{ \varepsilon |\nabla \eta |^{2}}{2} \, dx,
\end{equation*}
 we obtain
 \begin{equation*}
\begin{split}
 & \int_{\Omega}\nabla w \cdot \nabla \eta \, dx +\int_{\{w > \psi\} \cap \Omega}  f_{+}\,  \eta \,  dx-\int_{\{w < \psi\} \cap \Omega} f_{-}\, \eta \,dx +\\
  & \int_{\{w = \psi\} \cap \{\eta>0\}} f_{+}\eta \, dx-\int_{\{w = \psi\} \cap \{\eta<0\}} f_{-} \eta \,dx\geq 0,
 \end{split}
\end{equation*}
for every test function $\eta\in C_0^\infty(\Omega)$. Next the same calculation for $\varepsilon <  0 $ yields the reverse inequality. This implies  $\Delta w$    is a bounded distribution and thus is a measure. By
Radon-Nikodym,   we conclude there is  function in   $L^{1}(\Omega)$ that represents  $\Delta w \in   L^{1}(\Omega)$ as a distribution.
Thus, the following  holds in the weak sense
\begin{equation*}
\Delta w=f_+\chi_{\{w>\psi\}} - f_-\chi_{\{w<\psi\}} \text{ in } \{w\neq \psi\}
\end{equation*}
and
\begin{equation*}
f_+\chi_{\{w>\psi\}} - f_-\chi_{\{w\leq\psi\}} \leq \Delta w \leq f_+\chi_{\{w\geq\psi\}} - f_-\chi_{\{w<\psi\}} \text{ in } \Omega,
\end{equation*}
which implies   $w\in W^{2,p}(\Omega)$ for all $1\leq p<\infty$. Thus $w$ satisfies the following
\begin{equation}\label{eq:state1}
\left \{
\begin{array}{ll}
\triangle w= f_{+} \chi_{\{w >\psi\}}-f_{-} \chi_{\{w <\psi\}} +\Delta \psi \chi_{\{w=\psi\}} \ & \text{in}\ \Omega, \\
  w =g   &  \text{on}  \  \partial \Omega.\\
\end{array}
\right.
\end{equation}
Now let $ u=w-\psi$   then we have
\begin{equation}\label{eq:state2}
\left \{
\begin{array}{ll}
\Delta u  =(f_{+}- \Delta \psi) \chi_{\{u >0 \}}-(f_{-}+ \Delta \psi) \chi_{\{u <0\}}    &  \text {in} \, \Omega, \\
  u=g   &  \text{on} \,  \partial \Omega.\\
\end{array}
\right.
\end{equation}
Next choosing   $\phi= \Delta \psi$ yields (\ref{eq:state}).
\[
\]

\section{Existence and properties   of Minimizer}\label{farid}
\[
\]
In this section, we study some properties of control to state map. The following Lemma shows a type of monotonicity of solution with respect to coefficient $ \phi$.
\begin{lemma}\label{prop:nus1}
	Assume that $\phi_1 >  \phi_2.$  Let $u_i$ for $i=1,2$ be the solution of the following Dirichlet problem
	\begin{equation}\label{twophase}
	\left \{
	\begin{array}{ll}
	\Delta u_{i}= (f_{+}-\phi_{i})   \chi_{\{u_i >0 \}}- ( f_{-}+\phi_{i})  \chi_{\{u_{i} <0\}}  &  \text{in} \ \Omega, \\
	u_i=g\     & \text{on} \ \partial \Omega,
	\end{array}
	\right.
	\end{equation}
	Then
	\[
	u_1 \ge u_2  \quad in  \, \,  \Omega.
	\]
\end{lemma}
\begin{proof}  Let  $ \phi_{1} > \phi_{2}  > 0,$  we show  that  $u_{1}\geq u_{2}$. Set $D={\{x \in \Omega:u_{2} (x) > u_{1}(x)}\}.$
	If  $u_{2}(x)  \leq 0$  in $D$,    then $u_{1}(x)< 0$ and
	\[
	\Delta u_{1}  =- ( f_{-}+\phi_{1})  \leq  - ( f_{-}+\phi_{2})  =\Delta u_{2}.
	\]
	On the other hand, if  $u_{2} >0,$ then
	\[
	\Delta u_{2}=(f_{+}-\phi_{2})    \geq (f_{+}-\phi_{1}) \ge  \Delta u_1.
	\]
	Therefore in $D$ we have  $\Delta u_{2} \geq \Delta u_1$ and $u_2=u_1$ on the boundary of $D$   so by maximum principle $D=\emptyset$. Note that we have used the fact that $u_i$ is continuous.
\end{proof}
   Let
	\[
	M=\underset{x \in \Omega} {\text {ess sup}} (\phi_{1}(x) -\phi_{2}(x)).
	\]
	Now we claim that also $u_1+ M v\leq u_{2}$ in $B_1,$ where $v$ is the solution to
	\begin{equation*}
	\left \{
	\begin{array}{ll}
	\Delta v=1  &  \text{in} \, \Omega, \\
	v=0     & \text{on} \, \partial \Omega.
	\end{array}
	\right.
	\end{equation*}
	Set $\widetilde{D}={\{x \in \Omega:u_{2} (x) <  u_{1}(x)+ M v(x) }\}.$ Note that  $v(x) \le 0$ thus  in $\widetilde{D}$ the following holds 
	\[
	u_{2} (x) <  u_{1}(x)+ M v(x) \le  u_{1}(x).
	\]
	The inequalities above give
	\[
	\chi_{\{u_1 >0 \}}  \ge \chi_{\{u_2 >0 \}}, \quad   \quad   \chi_{\{u_1 < 0 \}}  \le \chi_{\{u_2 <0 \}}.
	\]
	Therefore	
	\begin{equation}\label{xx}
	\begin{split}
	\Delta (u_1 +M v)  = \Delta u_1 +M  & =  (f_{+}-\phi_{1}) \chi_{\{u_1 >0 \}} -  (f_{-}+\phi_{1}) \chi_{\{u_1 < 0 \}}     +M \\
	& \ge   (f_{+}-\phi_{1}) \chi_{\{u_2 >0 \}} - (f_{-}+\phi_{1}) \chi_{\{u_2 < 0 \}}+ M\\
	& \ge (f_{+}-\phi_{1} +M) \chi_{\{u_2 >0 \}} - (f_{-}+\phi_{1}- M) \chi_{\{u_2 < 0 \}}\\
	&    \ge (f_{+}-\phi_{2} ) \chi_{\{u_2 >0 \}} - (f_{-}+\phi_{2}) \chi_{\{u_2 < 0 \}}=\Delta u_2.\\
	\end{split}
	\end{equation}
	This shows on $\widetilde{D}$ we have $u_1+ M v \le  u_{2}$ which is contradiction thus  $\widetilde{D}=\emptyset.$	

The following theorem shows that the control-to-state operator $T$ is Lipschitz continuous.
\begin{theorem}\label{nus11}
Consider the map $u= T(\phi),$ then  $T: U_{ad} \rightarrow H^{1}(\Omega)$  is Lipschitz continuous.
\end{theorem}
\begin{proof}
 Assume $u_1= T(\phi_1), \, u_2= T(\phi_2).$ This means
 \begin{equation}\label{Jorge1}
 \Delta u_{1}= (f_{+}- \phi_{1} )\chi_{\{u_{1} > 0\}}   - (f_{-}+ \phi_{1} )\, \chi_{\{u_{1} < 0\}},
\end{equation}
 \begin{equation}\label{Jorge1}
 \Delta u_{2}= (f_{+}- \phi_{2} )\chi_{\{u_{2} > 0\}}   - (f_{-}+ \phi_{2} )\, \chi_{\{u_{2} < 0\}}.
\end{equation}
By subtracting these two equations, multiplying by $(u_2 - u_1)$, integrating, and subsequently applying integration by parts to the left-hand side 
\begin{equation}\label{xx}
\begin{split}
\int_{\Omega}| \nabla(u_{2} -  u_{1})|^2 \, dx  & = -\int_{\Omega} f_{+}\, ( \chi_{\{u_{2} > 0\}} - \chi_{\{u_{1} > 0\}})(u_2-u_1)dx\\
                                       &+ \int_{\Omega}   f_{-} \,  ( \chi_{\{u_{2} <  0\}} - \chi_{\{u_{1} < 0\}})  (u_2-u_1)\, dx\\
                                       & + \int_{\Omega}  \phi_2  \,( \chi_{\{u_{2} >   0\}} + \chi_{\{u_{2} < 0\}})(u_2-u_1)\, dx\\
                                        & - \int_{\Omega}  \phi_1 \, ( \chi_{\{u_{1} >   0\}} +   \chi_{\{u_{1} < 0\}}) (u_2-u_1)\, dx.
   \end{split}
   \end{equation}
   Note that
  \[
  ( \chi_{\{u_{2} > 0\}} - \chi_{\{u_{1} > 0\}})(u_2-u_1) \ge 0
  \]
    \[
     ( \chi_{\{u_{2} <  0\}} - \chi_{\{u_{1} < 0\}})(u_2-u_1) \le 0,
     \]
       to conclude that
\begin{equation}\label{xx}
\begin{split}
\int_{\Omega}| \nabla(u_{2} -  u_{1})|^2 \, dx   \le
 &  \int_{\Omega}  \phi_2  \,( \chi_{\{u_{2} >   0\}} + \chi_{\{u_{2} < 0\}})(u_2-u_1)\, dx\\
 & - \int_{\Omega}  \phi_1 \, ( \chi_{\{u_{1} >   0\}} +   \chi_{\{u_{1} < 0\}}) (u_2-u_1)\, dx.
    \end{split}
\end{equation}
Also, the following hold
\[
\chi_{\{u_{2} >   0\}} + \chi_{\{u_{2} < 0\}}\le 1,  \quad \chi_{\{u_{1} >   0\}} +   \chi_{\{u_{1} < 0\}}\le 1.
\]
This fact and the previous lemma  give
\begin{equation}\label{nus10}
\begin{split}
\int_{\Omega}| \nabla(u_{2} -  u_{1})|^2 \, dx & \le  \int_{\Omega} | \phi_{2}-   \phi_{1}| \, |u_{2}-u_{1}|\, dx 
  \end{split}
   \end{equation}
  Then Poincar\'{e} inequality, Young's inequality,    shows
   \[
   \| \nabla( u_{2} -  u_{1})\|_{L^{2}(\Omega)} \, dx  \le  C\|
   \phi_{2} -  \phi_{1}\|_{L^{2}(\Omega)} .
   \]
\end{proof}

Now we show that the optimal control problem (P) has a solution. Let $\{\phi_k\} $ be a minimizing sequence for Problem (P), i.e.,
\[
 J(\phi_k)\rightarrow\inf J.
  \]
  So it is bounded in $L^2(\Omega)$ and there is a weak  subsequence converging to some $\phi\in L^{2}(\Omega)$. Recall that $U_{ad}$ is convex, and therefore $\phi\in U_{ad}$.

Recall $g$ from the boundary condition in~\eqref{twophase1}.
If $u_k=T(\phi_k)$ satisfies \eqref{eq:state}, then $\Delta u_k$ is uniformly bounded. Indeed,
multiply  the equation
 \begin{equation}\label{Jorge51}
 \Delta u_{k}= (f_{+}- \phi_{k} )\chi_{\{u_{k} > 0\}}   - (f_{-}+ \phi_{k} )\, \chi_{\{u_{k} < 0\}},
\end{equation}
 by $(u_k -g)$  and integrate over $\Omega$  to obtain
   \begin{equation*}
 \int_{\Omega}   \nabla u_{k} \cdot  \nabla (u_{k}-g) \, dx = -\int_{\Omega}(f_{+}- \phi_{k} ) (u_{k}-g) \chi_{\{u_{k} > 0\}} \, dx    + \int_{\Omega} (f_{-}+ \phi_k) (u_{k}-g)  \chi_{\{u_{k} < 0\}} \, dx.
\end{equation*}
In the above, the harmonic extension of $g$ is still denoted by $g$.  Next, we have
\begin{equation}\label{xx}
\begin{split}
 \int_{\Omega}   | \nabla (u_{k}-g) |^{2}  \, dx  = & -\int_{\Omega}   \nabla g  \cdot  \nabla ( u_{k}-g) \, dx-\int_{\Omega}(f_{+}- \phi_{k} ) (u_{k}-g) \chi_{\{u_{k} > 0\}} \, dx  \\
   &  + \int_{\Omega} (f_{-}+ \phi_k) (u_{k}-g)  \chi_{\{u_{k} < 0\}} \, dx.
 \end{split}
\end{equation}
  Which implies
   \begin{equation*}
      \| \nabla (u_{k}-g)\|^{2}_{L^2}  \le  \frac{1}{2}\left(\|\nabla g\|^{2}_{L^2}+  \|\nabla ( u_{k}-g)\|^{2}_{L^2}\right)
      +\epsilon \|u_{k}-g\|^{2}_{L^2} + \frac 1{4\epsilon}\left(\|  f_{+}- \phi_{k} \|^{2}_{L^2} + \|f_{-}+ \phi_{k} \|^{2}_{L^2}\right).
\end{equation*}
Using Poincar\'{e}   inequality,
\[
  \| u_{k}-g\|^{2}_{L^2}  \le  \lambda^{-1}_{1}(\Omega)  \| \nabla (u_{k}-g)\|^{2}_{L^2},
  \]
where  $\lambda_{1}(\Omega)$ is the first eigenvalue of the Laplace operator, then rearranging terms we get for  $\epsilon <  \frac{\lambda_{1}(\Omega)}{2}$
\[
   \|  u_{k} \|_{W^{1,2}(\Omega)} \le C,  \quad \forall k.
\]
Hence,  there is $u\in W^{1,2}(\Omega)$  such that over a subsequence
\[
u_{k} \rightharpoonup  u \quad \text{ in }    W^{1,2}(\Omega),
\]
\[
u_{k}  \rightarrow   u \quad \text{ in }   L^{2}(\Omega),
\]
with $u=g$ on $\partial \Omega.$ If we consider a test function $\eta\in C_0^\infty(\Omega)$ and multiply \eqref{Jorge51}  by $\eta$, then we get
$$
-\int_{\Omega}\nabla u_k\cdot \nabla\eta\,dx=\int_{\Omega}(f_+-\phi_k)\eta\chi_{\{u_k>0\}}-(f_-+\phi_k)\eta\chi_{\{u_k<0\}}\,dx.
$$
Passing to the limit and recalling
the weak convergence $\phi_k\rightharpoonup \phi$ as well as the strong convergence $u_k\rightarrow u$ in $L^2(\Omega)$, this implies that
$$
-\int_{\Omega}\nabla u\cdot \nabla\eta\,dx=\int_{\Omega}(f_+-\phi)\eta\chi_{\{u>0\}}-(f_-+\phi)\eta\chi_{\{u<0\}}\,dx,
$$
and so $u=T(\phi)$.  Also by weak lower semi-continuity of a norm, we have
$$\|\phi\|_{L^2}\leq \liminf \|\phi_k\|_{L^2},$$
so,
$$
\inf J=\liminf \,J(\phi_k)=\liminf\left(\|u_k-z\|_{L^2}+\lambda \|\phi_k\|_{L^2}\right)\geq
\|u-z\|_{L^2}+\lambda \|\phi\|_{L^2}=J(\phi).
$$
Therefore,
\[
J(\phi)=\inf J.
\]
%

\section{Approximation of  the State Equation}
In general, we do not expect the solution operator of the obstacle problem to be differentiable, indeed see \cite{A15}, which proves this fact for the one-phase obstacle problem. In this section, we describe our strategy for regularizing the state equation. Set
 $  u^{\varepsilon}:=T^{\varepsilon}(\phi)$ to be the solution of the regularized problem
\begin{equation}\label{eq:state35}
\left \{
\begin{array}{ll}
\Delta u^{\varepsilon}= (f_{+}-  \phi) \chi^{\varepsilon}  (u^{\varepsilon}) - (f_{-}+  \phi) \chi^{\varepsilon}(-u^{\varepsilon})  &  \text {in } \ \Omega, \\
u^{\varepsilon}=g   &  \text{on}  \  \partial \Omega.
\end{array}
\right.
\end{equation}
Here  $\chi^{\varepsilon}(t) $ is a non-decreasing, smooth approximation of  $ \chi_{\{u>0\}} $ such that
\[
  \chi^{\varepsilon}(v)=
\left \{
\begin{array}{ll}
1 & v\geq\varepsilon, \\
0 & v\leq-\varepsilon. \\
\end{array}
\right.
\]
Without loss of generality,  we might choose
 \[
  \left \{
\begin{array}{ll}
 \chi^{\varepsilon}(u^{\varepsilon}) + \chi^{\varepsilon}(-u^{\varepsilon})=1, &\\
 \chi^{\varepsilon}  (0)=\frac{1}{2}.&
 \end{array}
\right.
\]
  Then  regularized state equation  (\ref{eq:state35}) becomes as
\begin{equation}\label{eq:state135}
\left \{
\begin{array}{ll}
\Delta u^{\varepsilon}= f_{+} \,  \chi^{\varepsilon}(u^{\varepsilon}) -  f_{-}\,  \chi^{\varepsilon}(-u^{\varepsilon}) -\phi  &  \text {in } \ \Omega, \\
u^{\varepsilon}=g   &  \text{on}  \  \partial \Omega.
\end{array}
\right.
\end{equation}
For simplicity, we rewrite     (\ref{eq:state35}) as
\begin{equation}\label{eq:state235}
\left \{
\begin{array}{ll}
\Delta u^{\varepsilon}=  \beta (u^{\varepsilon}) -\phi  &  \text {in } \ \Omega, \\
u^{\varepsilon}=g   &  \text{on}  \  \partial \Omega,
\end{array}
\right.
\end{equation}
where $ \beta (u^{\varepsilon})=  f_{+} \,  \chi^{\varepsilon}(u^{\varepsilon}) -  f_{-} \chi^{\varepsilon}(-u^{\varepsilon}).$ Also  let  $\Phi^{\varepsilon}(t)$ be the  regularization of    $\text {max}(v,0)$   defined  by
\[
 \Phi^{\varepsilon}(t)= \int_{-\infty}^{t} \chi^{\varepsilon}(s) \, ds.
 \]

\subsection{Properties of $T^{\varepsilon}$}

Next,  we show that problem (\ref{eq:state35}) has a unique solution.
\begin{theorem}\label{uniqu1}
For any $ \phi \in U_{ad}$ there exists a unique solution $ u^{\varepsilon} = T^\varepsilon(\phi)$ of (\ref{eq:state35}). Furthermore
   $u^{\varepsilon}\in W^{2,p}(\Omega)$ with $1\le p< \infty.$  Moreover, as $\varepsilon\to 0$,
\[
  u^{\varepsilon}\rightarrow u  \quad  \text{ strongly in } \quad  H^{1}(\Omega),
\]
where $u$ is the unique solution of \eqref{eq:state}.
\end{theorem}
\begin{proof}
Any solution $ u^{\varepsilon} $ of \eqref{eq:state35} can be considered a minimizer of  the functional
\[
I^{\varepsilon}(v)=\int_{\Omega}(\frac{1}{2}|\nabla v|^2+(f_{+}-  \phi)\Phi^{\varepsilon}(v)-(f_{-}+ \phi) \Phi^{\varepsilon}(-v))\,dx,
\]
which proves the existence by the direct method.   The uniqueness of the solution is a matter of using the comparison principle. Let $u_1^\varepsilon$ and $u_2^\varepsilon$ be two solutions of \eqref{eq:state35}. Let $ D=\{x\in \Omega : u_{1}^{\varepsilon}(x) \geq u_{2}^{\varepsilon}(x)\}$,
  by the monotonicity of $\chi^{\varepsilon}(u)$, we have $\Delta u_1^\varepsilon\geq\Delta u_2^\varepsilon$ in  $D$  as well as $u_1^\varepsilon=u_2^\varepsilon$ on $\partial  D$.
  Thus by the comparison principle $u_1^\varepsilon = u_2^\varepsilon$ in  $D$. Similarly, considering the set $\{u_1^\varepsilon\leq u_2^\varepsilon\}$ shows that $u_1^\varepsilon= u_2^\varepsilon$ in $\Omega$.

Since $u^{\varepsilon} $ is the weak solution  of \eqref{eq:state35} then by Calderon-Zygmund estimate (see, e.g., \cite{A7}), we infer that for any $K\subset\subset\Omega$
 \[
 \| u^{\varepsilon} \|_{W^{2,p}( K)}\leq C_{p,\Omega}\left(\| g\|_{L^{\infty}( \Omega )} +\|f_{+}-   \phi\|_{L^{\infty}}+\| f_{-}+  \phi\|_{L^{\infty}} \right).
 \]
 Thus the family  $\{u^{\varepsilon}\}$ is uniformly bounded in $ W^{2,p}(K)$   and also in $C^{1,\alpha}_{\text{loc}} (\Omega)$ for $p >  n$, by Sobolev embedding,   which indicates that $u^{\varepsilon}$ is the strong solution of 
 \eqref{eq:state35}.  Therefore  there exists a subsequence  $  \epsilon_{j} $   such that
 \begin{equation*}
 u ^{\epsilon_{j}} \rightarrow  u  \ \ \ \text{    in}  \    \  C^{1,\alpha}_{\textrm{loc}}(\Omega).
  \end{equation*}
Next, we multiply  the equation
 \begin{equation}\label{Jorge11}
 \Delta u^{\varepsilon}= (f_{+}- \phi)\chi^{\varepsilon}(u^{\varepsilon})   - (f_{-}+ \phi)\, \chi^{\varepsilon} (-u^{\varepsilon}),
\end{equation}
  by $(u^{\varepsilon}  -g)$  and integrate over $\Omega$  to obtain
   \begin{equation*}
 \int_{\Omega}   \nabla u^{\varepsilon}   \cdot  \nabla (u^{\varepsilon}-g) \, dx = -\int_{\Omega}(f_{+}- \phi  ) (u^{\varepsilon}-g) \, \chi^{\varepsilon}(u^{\varepsilon})\, dx    - \int_{\Omega} (f_{-}+ \phi ) (u^{\varepsilon}-g) \, \chi^{\varepsilon}(u^{\varepsilon}) \, dx,
\end{equation*}
  from here adding proper terms gives
\begin{equation}\label{xx}
\begin{split}
 \int_{\Omega}   | \nabla (u^{\varepsilon}-g) |^{2}  \, dx  &=  \int_{\Omega}   \nabla g  \cdot  \nabla (g- u^{\varepsilon}) \, dx\\
   &-\int_{\Omega}(f_{+}- \phi  ) (u^{\varepsilon}-g) \chi^{\varepsilon}(u^{\varepsilon}) \, dx    - \int_{\Omega} (f_{-}+ \phi ) (u^{\varepsilon}-g) \, \chi^{\varepsilon}(-u^{\varepsilon}) \, dx.
 \end{split}
\end{equation}
This implies
   \begin{equation*}
      \| \nabla (u^{\varepsilon}-g)\|^{2}_{L^2}  \le  \frac{1}{2}(\|\nabla g\|^{2}_{L^2}+  \|\nabla (g- u^{\varepsilon})\|^{2}_{L^2}+\|f_{+}- \phi  \|^{2}_{L^2}+ \|u^{\varepsilon}-g\|^{2}_{L^2}+
       \|f_{-}+ \phi  \|^{2}_{L^2}).
\end{equation*}
By rearranging terms and  using  Poincar\'{e}
\[
   \|  u^{\varepsilon} \|_{W^{1,2}(\Omega)} \le C,  \quad \textrm{uniformly  in}  \, \varepsilon.
   \]
Hence,   over the same  subsequence $\varepsilon=\varepsilon_{j}$  
\[
u^{\varepsilon} \rightharpoonup  u  \quad \text{in} \,    W^{1,2}(\Omega),
\]
\[
u^{\varepsilon}  \rightarrow   u \quad \text{in}  \,  L^{2}(\Omega),
\]
with $u=g$ on $\partial \Omega.$
By  lower semi-continuity it holds
 \[
 \int_{\Omega} |\nabla u|^{2} \, dx  \le \underset{\varepsilon \rightarrow 0}{\text{Liminf}}  \int_{\Omega} |\nabla  u^{\varepsilon}|^{2} \, dx
 \]
 and
\begin{equation}\label{xx}
\begin{split}
  &\int_{\Omega}(f_{+}-  \phi) \,   \text{max}(u^{\varepsilon} ,0)- (f_{-}+  \phi) \, \text {min}(u^{\varepsilon} ,0)\, dx     \rightarrow\\
  & \int_{\Omega}(f_{+}-  \phi) \,  \text{max}(\tilde{u},0)- (f_{-}+  \phi) \,  \text {min}(\tilde{u},0) \,dx.
  \end{split}
\end{equation}
 Now let
  \begin{equation}\label{functional}
I(v)=\int_{\Omega}\left(\frac{1}{2}|\nabla v|^2+ (f_{+} - \phi)     \text {max}(u,0)  + (f_{-} + \phi)     \text {min}(u,0)  \right)dx.
\end{equation}
 Then the following holds
\[
I(u)\le \underset{\varepsilon \rightarrow 0}{\text{Liminf}}\,  I^{\varepsilon}(u^\varepsilon)\le   \underset{\varepsilon \rightarrow 0} {\text{Liminf}} \, I^{\varepsilon}(v)=I(v),
\]
for  all admissible $v,$ which shows that $u$ is the minimizer. Since $u$ is unique then the whole sequence converges to $u$.  It is worth to note that  $u \in C^{1,\alpha}_{\textrm{loc}}(\Omega)$  so it  satisfies  the following  equations almost everywhere
\[
 \Delta u  = f_{+} -\phi \text {  in }    {\{u> 0}\},
 \]
 \[
 \Delta u = f_{-}+\phi  \text { in }   {\{u < 0}\},
 \]
 \[
   \Delta u =0  \ \  \text { a.e. in }  \  {\{u=0}\}.
 \]
\end{proof}

\begin{lemma}\label{uniqu3}
Consider the map   solution  $T^{\varepsilon}$ the solution of approximated equation  (\ref{eq:state35}). Then
   $ T^{\varepsilon}: {\phi \mapsto u^{\varepsilon}}$  is Lipschitz operator.
\end{lemma}
\begin{proof}
Let  $ u^{\varepsilon}=T^{\varepsilon}(\phi), \ v^{\varepsilon}=T^{\varepsilon}(\varphi)$. Then
\begin{equation}\label{F1}
\begin{split}
  \Delta(u^{\varepsilon}-v^{\varepsilon})=  & f_{+}(\chi^{\varepsilon}(u^{\varepsilon})-\chi^{\varepsilon}(v^{\varepsilon}))+
 f_{-}( \chi^{\varepsilon}(-v^{\varepsilon})-\chi^{\varepsilon}(-u^{\varepsilon}))-\\
 &-  \phi( \chi^{\varepsilon}(u^{\varepsilon})+\chi^{\varepsilon}(-u^{\varepsilon}))+  \varphi( \chi^{\varepsilon}(v^{\varepsilon})+\chi^{\varepsilon}(v^{\varepsilon})).
 \end{split}
\end{equation}
 Multiplying  by $(u^{\varepsilon}-v^{\varepsilon})$  and integrating   by parts  gives
\[
  \int_{\Omega} |\nabla (u^{\varepsilon}-v^{\varepsilon})|^{2} \, dx =
  \]
  \[
  \int_{\Omega}
(f_{+}[ \chi^{\varepsilon}(u)-\chi^{\varepsilon}(v)]+
 f_{-}[ \chi^{\varepsilon}(-v)-\chi^{\varepsilon}(-u)])\, (u^{\varepsilon}-v^{\varepsilon})\, dx +
 \]
 \[
  \int_{\Omega}(- \phi [ \chi^{\varepsilon}(u^{\varepsilon})+\chi^{\varepsilon}(-u^{\varepsilon})]+  \varphi[ \chi^{\varepsilon}(v^{\varepsilon})+\chi^{\varepsilon}(-v^{\varepsilon})] )(u^{\varepsilon}-v^{\varepsilon}) \, dx.
  \]
Note that $ (u^{\varepsilon}-v^{\varepsilon}) $ is vanishing on $ \partial \Omega.$ Also note that  $\chi^{\varepsilon}(u)$ is non decreasing then the followings hold
\[
( \chi^{\varepsilon}(u^{\varepsilon})-\chi^{\varepsilon}(v^{\varepsilon}))(u^{\varepsilon}-v^{\varepsilon}) \geq 0,
\]
and
\[
( \chi^{\varepsilon}(-v^{\varepsilon})-\chi^{\varepsilon}(-u^{\varepsilon}))(u^{\varepsilon}-v^{\varepsilon}) \geq 0.
\]
This implies
\begin{align*}
   \int_{\Omega} |\nabla (u^{\varepsilon}-v^{\varepsilon})|^{2} \, dx  \leq
   \int_{\Omega} | u^{\varepsilon}-v^{\varepsilon}|\,   |\phi-\varphi|\,  dx.
 \end{align*}
In the above, we used the facts  that
\[
 \chi^{\varepsilon}(u^{\varepsilon})+\chi^{\varepsilon}(-u^{\varepsilon}) = 1,
\]
\[
 \chi^{\varepsilon}(v^{\varepsilon})+\chi^{\varepsilon}(-v^{\varepsilon}) = 1.
 \]
Next using  Young's inequality and  Poincar\'{e} inequality  yields
\[
 \|u^{\varepsilon}-v^{\varepsilon}\|_{H^{1}_{0}}\leq  C\,  \| \psi-\varphi\|_{L^{2}}.
 \]
  \end{proof}

 $ T^{\varepsilon}$  is a good approximation of  $T $ in the  following sense.
 \begin{lemma}\label{uniqu4}
  Let $ \phi_{k}\rightarrow  \phi$  strongly  in $L^{2}(\Omega)$  then
    $u_{k}^{\varepsilon} \rightarrow u  $  strongly   in  $ H^{1}_{0}(\Omega),$  as $ \varepsilon$ tends to zero and $k$  tends to infinity,  where
    \begin{equation*}
      u_{k}^{\varepsilon}=T^{\varepsilon}(\phi_{k}).
      \end{equation*}
  \end{lemma}

\begin {proof}
\begin{equation}\label{F3}
\begin{split}
  & \| u^{\varepsilon}_{k} - u \|_{H^{1}_{0}} =\| T^{\varepsilon}(\phi_{k}) -T^{\varepsilon}(\phi) +T^{\varepsilon}(\phi) -T(\phi)\|_{H^{1}_{0}}\leq\\
  & \|T^{\varepsilon}(\phi_{k})-T^{\varepsilon}(\phi)\|_{H^{1}_{0}} +\|T^{\varepsilon}(\phi)-T(\phi)\|_{H^{1}_{0}}.
  \end{split}
\end{equation}
  Lemma (\ref{uniqu1}) implies that
\[
  \underset{\varepsilon \rightarrow 0}{\lim} \|T^{\varepsilon}(\phi)-T(\phi)\|=0,
  \]
  and by  Lemma (\ref{uniqu3}) we have
  \[
     \underset{ k \rightarrow  \infty}{\lim} \|T^{\varepsilon}(\phi_{k})-T^{\varepsilon}(\phi)\|=0.
\]
\end{proof}

In the next lemma, we will prove that the mapping  $T^{\varepsilon}$ is Fer\'chet-differentiable.
\begin{lemma}\label{Sunshine1}
The mapping  $ T^{\varepsilon}: U_{ad}\mapsto W^{2,2}(\Omega)$  is of class  $ C^{1}.$  Furthermore,   if $\phi, \psi  \in U_{ad}$,  and  $ \xi^{\varepsilon}=DT^{\varepsilon}(\phi) \cdot \psi$,   then  $ \xi^{\varepsilon}$ is the unique solution of the following Dirichlet problem
\begin{equation}\label{eq:state4}
\left \{
\begin{array}{ll}
  \Delta \xi^{\varepsilon} =  - \psi + \left[f_{+} \, (\chi^{\varepsilon} )^{'}(u) + f_{-} \,(\chi^{\varepsilon})^{'} (-u)\right]\xi^{\varepsilon}&  \textrm{in} \ \  \Omega, \\
  \xi^{\varepsilon}=0   &  \text{on} \ \  \partial \Omega.
\end{array}
\right.
\end{equation}
\end{lemma}
\begin{proof}
The proof follows from the implicit function theorem.  Consider the function
\[
F:D(\Delta) \times  L^\infty(\Omega)   \mapsto     L^{2}(\Omega)\times  L^{2}(\partial\Omega),
\]
defined by
\[
F(u,\phi)=\left( -\Delta u + (f_{+}-  \phi) \chi^{\varepsilon}  (u) - (f_{-}+  \phi) \chi^{\varepsilon}(-u), (u-g)_{\lfloor\partial\Omega} \right),
\]
where 
\[
D(\Delta)=\{u\in H^{1}(\Omega):  \Delta u \in L^{2}(\Omega)\}
\]
endowed with the norm $\| u\|_{L^{2}(\Omega)} + \|\Delta u\|_{L^{2}(\Omega)}$.  
 It is obvious that  $F$ is of class  $C^{1}$ and $F(T^\varepsilon(\phi),\phi)=(0,0)$ for every $ \phi \in U_{ad}$ by Lemma \ref{uniqu1}. Then $T^\varepsilon$ is $C^1$ and $\xi^{\varepsilon}=DT^{\varepsilon}(\phi)(\psi)$ satisfies
\[
\frac{\partial F}{\partial u}(u,\phi)(\xi^{\varepsilon})=\psi,
\]
which ends up with (\ref{eq:state4}). Note that we have used the relation
\[
\chi^{\varepsilon}(u)+\chi^{\varepsilon}(-u) = 1,
\]
in our calculation to derive \eqref{eq:state4}.
\end{proof}

In the rest of this section, we introduce a family of approximation problems for functional $J(\phi)$. Let  $ \varepsilon >0 $ and $\phi \in U_{ad} $ we introduce the functional
\begin{equation*}
 J^{\varepsilon}(\phi)= \frac{1}{2}\int_{\Omega}(u^{\varepsilon}(\phi)-z)^{2} \, dx+   \frac{\lambda}{2} \int_{\Omega}  |\phi|^{2}  \, dx.
   \end{equation*}
\textbf{Problem $(P^{\varepsilon})$}: Find $ \phi^{\varepsilon} \in U_{ad} $ such that\\
\begin{equation*}
J^{\varepsilon}(\phi^{\varepsilon})=\underset{s.t \,  u^{\varepsilon}=T^{\varepsilon}(\phi) }{\textrm{inf}} \ J^{\varepsilon}(\phi).
\end{equation*}
Similar to existence for problem $(P)$ we can prove the existence and uniqueness for problem $(P^{\varepsilon} ).$  Now we establish the following convergence result will be helpful in the next section.
\begin{theorem}\label{uniqu5}

For any   $\varepsilon > 0$, there is at least one pair of  solution $(\phi^{\varepsilon},  u^\varepsilon)$     for
the Problem ($P^{\varepsilon}$). Moreover, for  $ \varepsilon \rightarrow 0$, on a subsequence, we have
 $ (\phi^{\varepsilon}, u^{\varepsilon}) \rightarrow  (\phi^{*}, u^{*}) $ (an optimal pair for (P)) in the strong topology of   $L^{2}(\Omega) \times H^{1}(\Omega).$
\end{theorem}
 \begin {proof}
Let   $\varepsilon > 0$ be fixed and $(\phi_n, u_n) $ be  a minimizing sequence to  ($P^{\varepsilon}$). By assumption $\phi_n$ is bounded in $L^2$ so the two-phase equation implies that $u_n$ is bounded in $H^{2}(\Omega)$. We denote by $ (\phi^{\varepsilon}, u^{\varepsilon})$  some weak limits, on a subsequence, in the  $L^{2}(\Omega) \times H^{2}(\Omega)$ topology.  The weak lower semi-continuity of the functional $J^{\varepsilon}$  shows that this is an optimal pair. For every $\varepsilon > 0$  the above boundedness of the pair
  $(\phi^{\varepsilon}, u^{\varepsilon})$ remains valid. Let $(\overline{\phi}, \overline{u})$ be the weak limit on a subsequence of  $(\phi^{\varepsilon}, u^{\varepsilon})$ in the weak topology   $L^{2}(\Omega) \times H^{2}(\Omega)$.  Next, let $\phi \in U_{ad}$ be arbitrary. 
  
  \[
 \frac{1}{2}\int_{\Omega}(\overline{u} -z)^{2} \, dx+   \frac{\lambda}{2} \int_{\Omega}  |\overline{\phi}|^{2}  \, dx \le \text{liminf} \,( \frac{1}{2}\int_{\Omega}(u^{\varepsilon}(\phi^\varepsilon) -z)^{2} \, dx+   \frac{\lambda}{2} \int_{\Omega}  |\phi^\varepsilon|^{2}  \, dx )\le 
 \]
 \[
\le \text{liminf} \, \frac{1}{2}\int_{\Omega}(u^{\varepsilon}(\phi) -z)^{2} \, dx+   \frac{\lambda}{2} \int_{\Omega}  |\phi|^{2}  \, dx=\frac{1}{2}\int_{\Omega}(u(\phi) -z)^{2} \, dx+   \frac{\lambda}{2} \int_{\Omega}  |\phi|^{2}  \, dx.
    \]
In the above inequality, we used Lemma (\ref{uniqu1}). The uniqueness of the pair $(\phi^*, u^*)$  shows the second statement of Lemma holds. 
\end{proof}
\begin{proposition}
The functional  $J^{\varepsilon}(\phi)$ is of class  $C^{1}$. Moreover,
\[
DJ^{\varepsilon}(\phi)\cdot\psi=\int_{\Omega}(p^{\varepsilon}+ \lambda  \phi ) \,  \psi \, dx \quad   \forall\phi,\psi \in  U_{ad},
\]
where $ p^{\varepsilon} \in W^{2,p}(\Omega)$ is  the unique solution of the following  problem:
\begin{equation}\label{eq:state05}
\left \{
\begin{array}{ll}
- \Delta  p^{\varepsilon}+ \beta^{'}(u^{\varepsilon})p^{\varepsilon}=u^{\varepsilon}-z & \textrm{in}\, \Omega \\
p^{\varepsilon}=0   &  \textrm{on} \,  \partial \Omega.\\
\end{array}
\right.
\end{equation}
\end{proposition}
\begin{proof}
 By Lemma (\ref{Sunshine1}) and chain rule
\[
DJ^{\varepsilon} (\phi)  \cdot \varphi=\int_{\Omega} (u^{\varepsilon}-z)\xi^{\varepsilon} +\,   \lambda  \phi \, \psi\,  dx,
\]
where $ \xi^{\varepsilon}=DT^{\varepsilon}(\phi)\cdot \psi$ is as in Lemma  (\ref{Sunshine1}). Using (\ref {eq:state05}) in this expression we get
\begin{equation}\label{FS}
\begin{split}
&  DJ^{\varepsilon}(\phi)\cdot \psi= \int_{\Omega}\left[-\Delta  p^{\varepsilon} +  \beta^{'}(u^{\varepsilon})p^{\varepsilon}\right]\, \xi^{\varepsilon}+
 \lambda \phi \,  \psi \,  dx \\
&  =\int_{\Omega}(-  \Delta  \xi^{\varepsilon}+ \beta^{'}(u^{\varepsilon})\,\xi ^{\varepsilon}) p^{\varepsilon}+   \phi \,  \psi \, dx \\
& =\int_{\Omega}(p^{\varepsilon} +  \lambda\,   \phi) \,  \psi \,  dx. \\
 \end{split}
\end{equation}
The next Proposition  states the characterization of optimal pairs for problem  $ (P^{\varepsilon}) $:\begin{proposition}\label{uniqu8}
 \ Let  $(\phi^{\varepsilon},u^{\varepsilon})$ be  an optimal pair for $ (P^{\varepsilon}).$   Then there exists an adjoint function  $ p^{\varepsilon} \in H^{1}_{0}(\Omega)$, such that the triple $(\phi^{\varepsilon}, u^{\varepsilon}, p^{\varepsilon})$ satisfies the following system
\[
\left \{
\begin{array}{llll}
- \Delta  u^{\varepsilon}+ \beta(u^{\varepsilon})=    \phi^{\varepsilon} & \textrm{in}\ \Omega \\
-  \Delta   p^{\varepsilon}+ \beta^{'}(u^{\varepsilon})p^{\varepsilon}=u^{\varepsilon}-z & \text{in}\ \Omega \\
\int_{\Omega}(p^{\varepsilon} +  \lambda  \phi^{\varepsilon}) \,  \psi\geq0 & \forall \psi \in  U_{ad}\\
p^{\varepsilon}= \phi^{\varepsilon}=0    &  \text{on} \ \ \  \partial \Omega.
\end{array}
\right.
\]
\end{proposition}
\end{proof}
 
\section{Conclusion}
In this paper, we studied optimal control in the setting of a two-phase membrane. The structure of the forward PDE equations presents challenges in establishing the known results for variational inequality.  

We established the existence and uniqueness of the optimal pairs by formulating the problem as a minimization problem with appropriate regularization of the state equation.
  We proved that the control-to-state map \( T \) is Lipschitz continuous, ensuring the stability of the solutions concerning changes in the control.

 While the results are a promising step toward understanding the two-phase membrane problem, there are important considerations for future research. Along with studying more complex systems, such as the multi-phase problem or systems with multiple controls, appropriately convergent numerical techniques present a number of both conceptual methodological as well as practical challenges that will necessitate its own dedicated comprehensive investigation.

%
%
%
%
%


\renewcommand{\refname}{REFERENCES }

\end{document}